\documentclass[12pt, reqno]{amsart}
\usepackage{amsmath, amsthm, amscd, amsfonts, amssymb, graphicx, color}
\usepackage[bookmarksnumbered, colorlinks, plainpages]{hyperref}
\hypersetup{colorlinks=true,linkcolor=red, anchorcolor=green, citecolor=cyan, urlcolor=red, filecolor=magenta, pdftoolbar=true}

\newtheorem{theorem}{Theorem}[section]
\newtheorem{lemma}[theorem]{Lemma}

\newtheorem{corollary}[theorem]{Corollary}
\theoremstyle{definition}

\newtheorem{example}[theorem]{Example}

\newtheorem{remark}[theorem]{Remark}
\numberwithin{equation}{section}
\newcommand{\R}{{\mathbb R}}

\begin{document}

\setcounter{page}{1}

\title[Fefferman's Ineq. and Appl. in Elliptic PDEs]{Fefferman's Inequality and Applications in Elliptic Partial Differential Equations}

\author[Nicky K. Tumalun, Denny I. Hakim, \MakeLowercase{and} Hendra Gunawan]{Nicky K. Tumalun$^{1,*}$, Denny I. Hakim$^{2}$, \MakeLowercase{and} Hendra Gunawan$^3$}

\address{$^{1}$Analysis and Geometry Group, Bandung Institute of Technology, Jl. Ganesha No. 10, Bandung 40132, Indonesia}
\email{\textcolor[rgb]{0.00,0.00,0.84}{nickytumalun@yahoo.co.id}}

\address{$^{2}$Analysis and Geometry Group, Bandung Institute of Technology, Jl. Ganesha No. 10, Bandung 40132, Indonesia}
\email{\textcolor[rgb]{0.00,0.00,0.84}{dhakim@math.itb.ac.id}}

\address{$^{3}$Analysis and Geometry Group, Bandung Institute of Technology, Jl. Ganesha No. 10, Bandung 40132, Indonesia}
\email{\textcolor[rgb]{0.00,0.00,0.84}{hgunawan@math.itb.ac.id}}





\keywords{Morrey spaces; Stummel classes; Fefferman's inequality; Strong unique continuation property.
\newline \indent $^{*}$Corresponding author}


\begin{abstract}
In this paper we prove  Fefferman's inequalities associated to potentials belonging
to a generalized Morrey space $  L^{p,\varphi} $ or a Stummel class $ \tilde{S}_{\alpha,p} $.
Our results generalize and extend  Fefferman's inequalities obtained in \cite{CRR,CF,F,Z1}.
We also show that the logarithmic of non-negative weak solution of second order elliptic partial differential equation,
where its potentials are assumed in generalized Morrey spaces and Stummel classes, belongs to the bounded mean oscillation class.
As a consequence, this elliptic partial differential equation has the strong unique continuation property.
An example of an elliptic partial differential equation where its potential belongs to certain Morrey spaces
or Stummel classes which does not satisfy the strong unique continuation is presented.
\end{abstract} \maketitle

\section{Introduction and Statement of Main Results}\label{intro}

Let $ 1 \leq p < \infty $ and $ \varphi : (0,\infty) \rightarrow (0,\infty)$.
The \textbf{generalized Morrey space} $L^{p,\varphi} := L^{p,\varphi}(\mathbb{R}^{n}) $,
which was introduced by Nakai in \cite{N}, is the
collection of all functions $ f \in L^{p}_{\rm loc}(\mathbb{R}^{n}) $ satisfying
\begin{equation*}
\| f \|_{L^{p,\varphi}} := \sup_{x \in \mathbb{R}^{n}, r>0} \left( \frac{1}{\varphi(r)}
\int\limits_{|x-y|<r} |f(y)|^{p} dy \right)^{\frac{1}{p}} < \infty.
\end{equation*}
Note that $ L^{p,\varphi} $ is a Banach space with norm $ \| \cdot \|_{L^{p,\varphi}}$.
If $ \varphi (r) = 1$, then $ L^{p,\varphi} = L^{p}$. If $ \varphi (r) = r^{n}$, then
$ L^{p,\varphi} = L^{\infty}$. If $ \varphi (r) = r^{\lambda} $ where $ 0 < \lambda < n$,
then $ L^{p,\varphi} = L^{p,\lambda} $ is the classical Morrey space introduced in \cite{M}.

We will assume the following conditions for $ \varphi $ which will be stated whenever necessary.

\begin{itemize}
	\item [(1)]
	There exists $ C>0 $ such that
	\begin{equation}\label{a1}
	s \leq t \Rightarrow \varphi(s) \leq C \varphi(t).
	\end{equation}
	We say $ \varphi $ \textbf{almost increasing} if $\varphi$ satisfies this condition.
	
	\item [(2)]
	There exists $ C>0 $ such that
	\begin{equation}\label{a2}
	s \leq t \Rightarrow \frac{\varphi(s)}{s^{n}} \geq C \frac{\varphi(t)}{t^{n}}.
	\end{equation}
	We say $ \varphi(t)t^{-n} $ \textbf{almost decreasing} if $ \varphi(t)t^{-n} $ satisfies this condition.
	
	\item[(3)]
	For $ 1 < \alpha < n $, $ 1 < p < \frac{n}{\alpha}$, there exists a constant $ C>0 $ such that for every  $ \delta > 0$,
	\begin{equation}\label{a3}
	\int\limits_{\delta}^{\infty} \frac{\varphi(t)}{t^{(n+1) - \frac{p}{2}(\alpha+1) }} \, dt \leq C \delta^{\frac{p}{2}(1-\alpha)}.
	\end{equation}
	This last condition for $ \varphi $ will be called \textbf{Nakai's condition}.
\end{itemize}
One can check that the function $ \varphi(t) = t^{n-\alpha p} $, $ t>0 $, satisfies all conditions
\eqref{a1}, \eqref{a2}, and \eqref{a3}. Moreover, for a non-trivial example, we
have the function $ \varphi_{0}(t) = \log(\varphi(t)+1) = \log(t^{n-\alpha p}+1) $, $ t>0 $, which satisfies all conditions above.

Let $ M $ be the \textbf{Hardy-Littlewood maximal operator}, defined by
\begin{equation*}
M(f)(x) := \sup_{r>0} \frac{1}{|B(x,r)|} \int\limits_{B(x,r)} |f(y)| \, dy
\end{equation*}
for every $ f \in L^{1}_{\rm loc}(\mathbb{R}^{n}) $. The function $ M(f) $ is called the
\textbf{Hardy-Littlewood maximal function}. Notice that, for every $ f \in L^{p}_{\rm loc}(\mathbb{R}^{n}) $
where $ 1 \leq p \leq \infty$, $ M(f)(x) $ is finite for almost all $ x \in \mathbb{R}^{n}$.
Using Lebesgue Differentiation Theorem, we have
\begin{equation}\label{f-Mf}
|f(x)| = \lim\limits_{r \rightarrow 0} \frac{1}{|B(x,r)|} \int_{B(x,r)} |f(y)| \, dy
\leq \lim\limits_{r \rightarrow 0} M(f)(x) = M(f)(x),
\end{equation}
for every $ f \in L^{1}_{\rm loc}(\mathbb{R}^{n}) $ and for almost all $ x \in \mathbb{R}^{n}$.
Furthermore, for every $ f \in L^{p}_{\rm loc}(\mathbb{R}^{n}) $ where $ 1 \leq p \leq \infty $
and $ 0 < \gamma < 1$, the nonnegative function
$ w(x) := \left[ M(f)(x) \right]^{\gamma} $ is an $ A_{1} $ weight, that is,
\begin{equation*}
M(w)(x) \leq C(n,\gamma) w(x).
\end{equation*}
These maximal operator properties can be found in \cite{GR,S}.

We will need the following property about the boundedness of
the Hardy-Littlewood maximal operator on generalized Morrey spaces $ L^{p,\varphi}$ which is stated in \cite{N,N2,YS}, that is,
\begin{equation*}
\| M(f) \|_{L^{p,\varphi}} \leq C(n,p) \| f \|_{L^{p,\varphi}},
\end{equation*}
for every $f \in L^{p,\varphi}$, where $ 1 \leq p < \infty$ and $ \varphi $ satisfies conditions \eqref{a1} and \eqref{a2}.
 Note that in \cite{N}, the proof of this boundedness property relies on a condition about the integrability
of $ \varphi(t)t^{-(n+1)} $ over the interval $ (\delta,\infty) $ for every positive number $ \delta$.

Let $ 1 \leq p < \infty $ and $0<\alpha<n$. For $ V \in L^{p}_{\rm loc}(\mathbb{R}^{n})$,
we write
\begin{equation*}
\eta_{\alpha,p}V(r) := \sup_{x \in \mathbb{R}^{n}} \left(\int\limits_{|x-y|<r}
\frac{|V(y)|^{p}}{|x-y|^{n-\alpha}} \, dy \right)^{\frac{1}{p}},\quad r>0.
\end{equation*}
We call $\eta_{\alpha,p}V$ the \textbf{Stummel $p$-modu\-lus} of $V$. If $\eta_{\alpha,p}V(r)$
is finite for every $ r>0 $, then $\eta_{\alpha,p}V(r)$ is nondecreasing on the set of
positive real numbers and satisfies
\begin{equation*}
\eta_{\alpha,p}V(2r) \leq C(n,\alpha) \, \eta_{\alpha,p}V(r),\quad r>0.
\end{equation*}
The last inequality is known as the \textbf{doubling condition} for the Stummel $p-$modulus of $ V $ \cite[p.550]{THG}.

For each $0<\alpha<n$ and $1\le p<\infty$, let
\begin{equation*}
\tilde{S}_{\alpha,p}:= \left\lbrace V \in L^{p}_{\rm loc}(\mathbb{R}^{n}) :
\eta_{\alpha,p}V(r) < \infty \, \text{for all} \, r>0 \right\rbrace
\end{equation*}
and
\begin{equation*}
S_{\alpha,p}:= \left\lbrace V \in L^{p}_{\rm loc}(\mathbb{R}^{n}) :
\eta_{\alpha,p}V(r) < \infty \,\, \text{for all} \, r>0 \, \text{and} \,
\lim\limits_{r\rightarrow 0}\eta_{\alpha,p}V(r) = 0 \right\rbrace.
\end{equation*}
The set $S_{\alpha,p}$ is called a \textbf{Stummel class}, while
$\tilde{S}_{\alpha,p}$ is called a \textbf{bounded Stummel modulus class}.
For $p=1$, $S_{\alpha,1} := S_{\alpha}$ are the Stummel classes which
were introduced in \cite{DH,RZ}. We also write $\tilde{S}_{\alpha,1} :=
\tilde{S}_{\alpha}$ and $\eta_{\alpha,1}:=\eta_{\alpha}$. It was shown in
\cite{THG} that $\tilde{S}_{\alpha,p}$ contains $S_{\alpha,p}$ properly.
These classes play an important role in studying the regularity theory of
partial differential equations (see \cite{CRR,CFG,DH,S,Z1} for example), and have an inclusion
relation with Morrey spaces under some certain conditions \cite{CRR,RZ,TG,THG}.

Now we state our results for  Fefferman's inequalities:

\begin{theorem}\label{th 2.1}
Let $ 1 < \alpha < n$, $ 1 < p < \frac{n}{\alpha}$, and $ \varphi $ satisfy conditions \eqref{a1}, \eqref{a2}, \eqref{a3}.
If  $ V \in L^{p,\varphi}$, then
	\begin{equation}\label{feff-in Morrey}
	\int\limits_{\mathbb{R}^{n}} |u(x)|^{\alpha} |V(x)| dx
	\leq C \| V \|_{L^{p,\varphi}} \int\limits_{\mathbb{R}^{n}} |\nabla u(x)|^{\alpha} dx
	\end{equation}
for every $ u \in C^{\infty}_{0}(\mathbb{R}^{n})$.
\end{theorem}

\begin{theorem}\label{th 2.2}
Let $1\le p<\infty$, $1\le \alpha\le 2$, and $ \alpha < n$. If $V\in \tilde{S}_{\alpha,p}(\mathbb{R}^n)$, then
there exists a constant $C:=C(n,\alpha)>0$ such that
	\begin{align*}
	\int\limits_{B(x_0, r_0)}
	|V(x)|^p |u(x)|^\alpha  \ dx
	\le
	C
	[\eta_{\alpha,p}V(r_0)]^p \int\limits_{B(x_0, r_0)} |\nabla u(x)|^\alpha \ dx,
	\end{align*}
for every ball $B_0:=B(x_0, r_0) \subseteq \mathbb{R}^n$ and $u\in C^\infty_{0}(\mathbb{R}^n)$ with ${\rm supp}(u)\subseteq B_0$.
\end{theorem}

\begin{remark}
The assumption  that the function $ u $ belongs to $ C^{\infty}_{0}(\mathbb{R}^{n}) $ in Theorem \ref{th 2.1} and
Theorem \ref{th 2.2} can be weakened by the assumption that $ u $ has a weak gradient in a ball
$ B \subset \mathbb{R}^{n} $ and a compact support in $ B $ (see \cite[p.480]{WZ}).
\end{remark}

In 1983, C. Fefferman \cite{F} proved Theorem \ref{th 2.1} for the case $ V \in L^{p,n-2p}$, where $ 1 < p \leq \frac{n}{2} $.
The inequality \eqref{feff-in Morrey} is now known as \textbf{Fefferman's inequality}. Chiarenza and Frasca \cite{CF} extended
the result \cite{F} by proving Theorem \ref{th 2.1} under the assumption that $ V \in L^{p,n-\alpha p} $, where $ 1 < \alpha < n $
and $ 1 < p \leq \frac{n}{\alpha}$. By setting
$ \varphi(t) = t^{n-\alpha p} $ in Theorem \ref{th 2.1}, we can recover the result in \cite{CF} and \cite{F}.

For the particular case  where $ V \in \tilde{S}_{2} $, Theorem \ref{th 2.2} is proved by Zamboni \cite{Z1},
and can be also concluded by applying the result Fabes \textit{et al.} in \cite[p.197]{FGL} with an additional assumption
that $ V $ is a radial function. Recently, this  result is reproved in  \cite{CRR}. Although $ \tilde{S}_{\alpha} \subset \tilde{S}_{2} $ whenever $ 1 \leq \alpha \leq 2 $ \cite[p.553]{THG}, the authors still do not know how to deduce Theorem \ref{th 2.2} from this result.

It must be noted that  Theorems \ref{th 2.1} and  \ref{th 2.2} are independent each other,
which means  that $ L^{p,n-\alpha p}$, where $ 1 < \alpha < n $ and $ 1 < p \leq \frac{n}{\alpha}$,
is not contained in $ S_{\alpha,p} $. Conversely,  $S_{\alpha,p}$ is not contained in $ L^{p,n-\alpha p}$.
Indeed, if we define $ V_{1}: \mathbb{R}^{n} \rightarrow \mathbb{R} $ by the formula $ V_{1}(y):=|y|^{-\alpha}$,
then $ V_{1} \in L^{p,n-\alpha p} $, but $ V_{1} \notin \tilde{S}_{\alpha,p}$. For the function
$ V_{2}: \mathbb{R}^{n} \rightarrow \mathbb{R} $ which is defined by the formula $ V_{2}(y):=|y|^{-\frac{1}{p}}$,
we have $ V_{2} \in \tilde{S}_{\alpha,p}$, but $ V_{2} \notin L^{p,n-\alpha p}$.

In order to apply Theorems \ref{th 2.1} and \ref{th 2.2}, let us recall the following definitions. Let $\Omega$ be an open and bounded subset of $\mathbb{R}^{n}$.
Recall that the \textbf{Sobolev space} $ H^1(\Omega) $ is the set of all functions $u\in L^2(\Omega)$
for which $\displaystyle \frac{\partial u}{\partial x_i} \in L^2(\Omega)$ for all $i=1, \ldots, n$, and is equipped by the \textbf{Sobolev norm}
\begin{equation*}
\| u \|_{H^1(\Omega)} = \| u \|_{L^2(\Omega)} + \sum_{i= 1}^{n} \left\|  \frac{\partial u}{\partial x_i} \right\|_{L^2(\Omega)}.
\end{equation*}
The closure of $u \in C_{\rm 0}^{\infty}(\Omega)$ in $ H^1(\Omega) $ under the Sobolev norm is denoted by $ H^1_{0}(\Omega) $.

Define the operator $L$ on $H^1_{0}(\Omega)$ by
\begin{equation}\label{b1}
Lu:= -\sum_{i, j = 1}^{n} \frac{\partial}{\partial x_{i}} \left( a_{ij} \frac{\partial u}{\partial x_{j}}
\right) + \sum_{i= 1}^{n} b_{i}\frac{\partial u}{\partial x_{i}} + Vu
\end{equation}
where $a_{ij} \in L^{\infty}(\Omega) $, $b_{i}$ ($ i,j=1, \ldots, n $) and $V$ are real valued measurable functions on $ \mathbb{R}^{n} $.
Throughout this paper, we assume that the matrix $ a(x) := (a_{ij}(x)) $ is symmetric
on $\Omega$ and that the ellipticity and boundedness conditions
\begin{equation}\label{b2}
\lambda |\xi|^{2} \leq \sum_{i, j = 1}^{n} a_{ij}(x) \xi_{i}\xi_{j} \leq \lambda^{-1} |\xi|^{2}
\end{equation}
hold for some $\lambda > 0$, for all $\xi \in \mathbb{R}^{n}$, and for almost all $ x \in \Omega $.

There are two assumptions on the potentials of the operator $L$ \eqref{b1} in this paper. We assume either:
\begin{equation}\label{a.4}
\begin{cases}
& \varphi \,\, \text{satisfies  \eqref{a1}, \eqref{a2}, \eqref{a3} ($ 1 < \alpha \leq 2 $)}, \\
& b_{i}^{2} \in L^{p,\varphi}, \, i=1, \ldots, n, \\
& V \in  L^{p,\varphi} \cap L^{2}_{\rm loc}(\mathbb{R}^{n}),
\end{cases}
\end{equation}
or,
\begin{equation}\label{a.5}
\begin{cases}
& 1 \leq \alpha \leq 2, \\
& b_{i}^{2} \in \tilde{S}_{\alpha}, \, i=1, \ldots, n, \\
& V \in \tilde{S}_{\alpha}.
\end{cases}
\end{equation}

We say that $u \in H^{1}_{0}(\Omega)$ is a \textbf{weak solution} of the equation
\begin{equation}\label{b3}
Lu = 0
\end{equation}
if
\begin{equation}\label{b3-b}
\int\limits_{\Omega} \left(  \sum_{i, j = 1}^{n} a_{ij} \frac{\partial u}{\partial x_{i}}
\frac{\partial \psi}{\partial x_{j}} + \sum_{i= 1}^{n} b_{i} \frac{\partial u}{\partial x_{i}}
\psi + Vu \psi \right)dx = 0.
\end{equation}
for all $ \psi \in H^{1}_{0}(\Omega)$ (see the definition in \cite{CRR,E,Z1}).
Note that, for the case $ \alpha = 2$, the equation \eqref{b3} was considered in \cite{CRR,Z1}.
If we choose $ b_{i} = 0$ for all $ i = 1, \dots, n$, then \eqref{b3} becomes the Schr\"{o}dinger equation \cite{CFG}.

A locally integrable function $ f $ on $ \mathbb{R}^{n} $ is said to be of
\textbf{bounded mean oscillation} on a ball $ B \subseteq \mathbb{R}^{n} $ if there is a
constant $ C>0 $ such that for every ball $ B' \subseteq B $,
\begin{equation*}
\frac{1}{|B'|} \int\limits_{B'} |f(y)-f_{B'}|dy \leq C.
\end{equation*}
We write $ f \in BMO(B) $ if $ f $ is of bounded mean oscillation on $ B$.
Moreover, if $ 1 \leq \alpha < \infty $ and there is a
constant $ C>0 $ such that for every ball $ B' \subseteq B$,
\begin{equation*}
\left( \frac{1}{|B'|} \int\limits_{B'} |f(y)-f_{B'}|^{\alpha}dy \right)^{\frac{1}{\alpha}}  \leq C,
\end{equation*}
we write $ f \in BMO_{\alpha}(B)$. By using H\"{o}lder's inequality and the John-Nirenberg theorem (we refer to \cite{K} for more detail about this John-Nirenberg Theorem), we can prove that $ BMO_{\alpha}(B) = BMO(B) $.

As an application  Theorems \ref{th 2.1} and  \ref{th 2.2} to equation $ Lu=0 $ \eqref{b3}, we have the following result.

\begin{theorem}\label{th 3.1}
Let $u \geq 0$ be a weak solution of $Lu=0$ and $B(x,2r) \subseteq \Omega$ where $ r \leq 1$.
Then there exists a constant $C>0$ such that
\begin{equation*}
	\frac{1}{|B(x,r)|} \int\limits_{B(x,r)}\left| \log(u+\delta) - \log(u+\delta)_{B(x,r)} \right|^{\alpha}dy \leq C,
\end{equation*}
for every $ \delta > 0 $.
\end{theorem}

Theorem \ref{th 3.1} tells us that $ \log(u+\delta) \in BMO_{\alpha}(B) $, where $ B $ is an open ball which is contained in $ \Omega $, and
$ u $ is the non-negative weak solution of equation $ Lu =0 $, given by \eqref{b3}. Under the assumptions that $b^{2}_{i} $ and
$ V $ belong to $ \tilde{S}_{2} $, Theorem \ref{th 3.1} is obtained in \cite{CFG} for the case Schr\"{o}dinger equation ($ b_{i} = 0 $) and
in \cite{Z1,CRR} for the case $ Lu = 0 $ ($ b_{i} \neq 0 $).
To the best of our  knowledge, the assumptions in \eqref{a.4} have never been used
for proving Theorem \ref{th 3.1} as well as the assumption $ \alpha \in \left[ 1,2\right)  $ in \eqref{a.5}.

Let $ w \in L^{1}_{\rm loc}(\Omega) $ and $ w \geq 0 $ in $ \Omega $. The function $w$
is said to \textbf{vanish of infinite order} at $x_{0} \in \Omega$ if
\begin{equation*}
\lim_{r \rightarrow 0} \frac{1}{|B(x_{0},r)|^{k}} \int\limits_{B(x_{0},r)} w(x)dx = 0, \qquad \forall k>0.
\end{equation*}
The equation $ Lu = 0$, which is given in \eqref{b3}, is said to have the \textbf{strong unique continuation property} in $ \Omega $
if for every nonnegative solution $ u $ which vanishes of infinite order at some $x_{0} \in \Omega$, then $ u \equiv 0 $ in
$ B(x_{0},r) $ for some $ r>0 $. See this definition, for example in \cite{GL,JK,KT}.

Theorem \ref{th 3.1} gives the  following result.

\begin{corollary}\label{cor 2}
The equation $ Lu=0 $ has the strong unique continuation property in $ \Omega$.
\end{corollary}

This strong unique continuation property was studied by several authors. For example, Chiarenza and Garofalo in \cite{CF}
discussed the Schr\"{o}dinqer inequality of the form $ Lu+ Vu \geq 0$, where the potential $ V $ belongs to Lorentz spaces $ L^{\frac{n}{2},\infty} (\Omega)$. For the differential inequality of the form $ | \Delta u | \leq |V| |u|  $ where its potential
also belong to $ L^{\frac{n}{2}}(\Omega)$, see Jerison and Kenig \cite{JK}. Garofalo and Lin \cite{GL} studied the equation \eqref{b3} where the potentials are bounded by certain functions.

Fabes \textit{et al.} studied the strong unique continuation property for Schr\"{o}dinqer equation $ -\Delta u + Vu = 0 $,
where the assumption for $ V $ is radial function in $ S_{2} $ \cite{FGL}. Meanwhile, Zamboni \cite{Z1} and Castillo \textit{et al.} \cite{CRR} also studied the equation \eqref{b3} under the assumption that the potentials belong to $ S_{2} $. At the end of this paper, we will give an example
of Schr\"{o}dinqer equation $ -\Delta u + Vu = 0 $ that does not satisfy the strong unique continuation property, where $ V \in L^{p,n-4p} $ or
$ V \in \tilde{S}_{\beta} $ for all $ \beta \geq 4 $.

\section{Proofs}\label{sec:1}

In this section, we prove  Fefferman's inequalities, which we state as Theorems \ref{th 2.1} and  \ref{th 2.2}
above. First, we start with the case where the potential belongs to a generalized Morrey spaces. 
Second, we consider the potential from a Stummel class. Furthermore, we present an inequality which is deduced from this inequality.

\subsection{Fefferman's Inequality in Generalized Morrey Spaces}

We start with the following lemma for potentials in generalized Morrey spaces.

\begin{lemma}\label{lem 2.1}
	Let $1<p<\infty$ and $ \varphi $ satisfy the conditions \eqref{a1} and \eqref{a2}. If $ 1 < \gamma < p $ and $ V \in L^{p,\varphi}$, then
	$ [M(|V|^{\gamma})]^{\frac{1}{\gamma}} \in A_{1} \cap L^{p,\varphi}$.
\end{lemma}

\begin{proof}
	According to our discussion above, $ [M(|V|^{\gamma})]^{\frac{1}{\gamma}} \in A_{1}$.
	Using the boundedness of the Hardy-Littlewood maximal operator on generalized Morrey spaces, we have
	\begin{equation*}
	\| [M(|V|^{\gamma})]^{\frac{1}{\gamma}} \|_{L^{p,\varphi}} = \| M(|V|^{\gamma})
	\|_{L^{\frac{p}{\gamma},\varphi}}^{\frac{1}{\gamma}}
	\leq C \| V \|_{L^{p,\varphi}} < \infty.
	\end{equation*}
	Therefore $ [M(|V|^{\gamma})]^{\frac{1}{\gamma}} \in L^{p,\varphi} $.
\end{proof}

Using Lemma \ref{lem 2.1}, we obtain the following property.

\begin{lemma}\label{lem 2.2}
	Let $ \varphi $ satisfy the conditions \eqref{a1}, \eqref{a2}, and \eqref{a3}.
	If $ V \in L^{p,\varphi}$, then
	\begin{equation*}
	\int_{\mathbb{R}^{n}} \frac{|V(y)|}{|x-y|^{n-1}} \, dy
	\leq C(n,\alpha,p) \| V \|_{L^{p,\varphi}}^{\frac{1}{\alpha}} [M(V)(x)]^{\frac{\alpha-1}{\alpha}}.
	\end{equation*}
\end{lemma}

\begin{proof}
	Let $ \delta > 0$. Then
	\begin{equation}\label{in lem 2.1-1}
	\int_{\mathbb{R}^{n}} \frac{|V(y)|}{|x-y|^{n-1}} \, dy
	= \int_{|x-y|<\delta} \frac{|V(y)|}{|x-y|^{n-1}} \, dy + \int_{|x-y| \geq \delta} \frac{|V(y)|}{|x-y|^{n-1}} \, dy.
	\end{equation}
	Using Lemma (a) in \cite{H}, we have
	\begin{equation}\label{in lem 2.1-2}
	\int_{|x-y|<\delta} \frac{|V(y)|}{|x-y|^{n-1}} \, dy \leq C(n) M(V)(x) \delta.
	\end{equation}
	For the second term on the right hand side \eqref{in lem 2.1-1}, let $ q = n - \frac{p}{2}(\alpha+1)$,
	we use H\"{o}lder's inequality to obtain
	\begin{align}\label{in lem 2.1-3}
	\int_{|x-y| \geq \delta} \frac{|V(y)|}{|x-y|^{n-1}} \, dy
	& = \int_{|x-y| \geq \delta} \frac{|V(y)||x-y|^{ \frac{q}{p} + 1-n}}{|x-y|^{\frac{q}{p}}} \, dy \nonumber \\
	& \leq \left( \int_{|x-y| \geq \delta} \frac{|V(y)|^{p}}{|x-y|^{q}} \, dy \right)^{\frac{1}{p}} \nonumber \\
	& \qquad \times \left( \int_{|x-y| \geq \delta} |x-y|^{ (\frac{q}{p} + 1-n)(\frac{p}{p-1}) }  dy \right)^{\frac{p-1}{p}}.
	\end{align}
	Note that Nakai's condition gives us
	\begin{align}\label{in lem 2.1-4}
	\int_{|x-y| \geq \delta} \frac{|V(y)|^{p}}{|x-y|^{q}} \, dy
	& = \sum_{k=0}^{\infty} \int_{2^{k} \delta \leq |x-y| < 2^{k+1} \delta} \frac{|V(y)|^{p}}{|x-y|^{q}} \, dy \nonumber \\
	& \leq C \| V \|_{L^{p,\varphi}}^{p}   \int_{\delta}^{\infty} \frac{\varphi(t)}{t^{q+1}} \, dt \nonumber \\
	& \leq C \| V \|_{L^{p,\varphi}}^{p} \delta^{n-p \alpha -q}.
	\end{align}
	Since $ n + (\frac{q}{p} + 1-n)(\frac{p}{p-1}) < 0$, we obtain
	\begin{align}\label{in lem 2.1-5}
	\int_{|x-y| \geq \delta} |x-y|^{ (\frac{q}{p} + 1-n)(\frac{p}{p-1}) }  dy
	& = C(n,p,\alpha) \delta^{n + (\frac{q}{p} + 1-n)(\frac{p}{p-1})}.
	\end{align}
	Introducing \eqref{in lem 2.1-4} and \eqref{in lem 2.1-5} in \eqref{in lem 2.1-3}, we have
	\begin{align}\label{in lem 2.1-6}
	\int_{|x-y| \geq \delta} \frac{|V(y)|}{|x-y|^{n-1}} \, dy
	& \leq C \| V \|_{L^{p,\varphi}} \left(  \delta^{n-p\alpha -q} \right)^{\frac{1}{p}}
	\left(   \delta^{n + (\frac{q}{p} + 1-n)(\frac{p}{p-1})}  \right)^{\frac{p-1}{p}} \nonumber \\
	& =  C \| V \|_{L^{p,\varphi}} \delta^{1-\alpha}.
	\end{align}
	From \eqref{in lem 2.1-6}, \eqref{in lem 2.1-2} and \eqref{in lem 2.1-1}, we get
	\begin{align}\label{in lem 2.1-7}
	\int_{\mathbb{R}^{n}} \frac{|V(y)|}{|x-y|^{n-1}} \, dy
	\leq C M(V)(x) \delta + C \| V \|_{L^{p,\varphi}} \delta^{1-\alpha}
	\end{align}
	For $ \delta = \| V \|_{L^{p,\varphi}}^{\frac{1}{\alpha}} [M(V)(x)]^{-\frac{1}{\alpha}}$,
	the inequality \eqref{in lem 2.1-7} becomes
	\begin{align}
	\int_{\mathbb{R}^{n}} \frac{|V(y)|}{|x-y|^{n-1}} \, dy
	\leq C [M(V)(x)]^{1-\frac{1}{\alpha}} \| V \|_{L^{p,\varphi}}^{\frac{1}{\alpha}}
	= C [M(V)(x)]^{\frac{\alpha-1}{\alpha}} \| V \|_{L^{p,\varphi}}^{\frac{1}{\alpha}}. \nonumber
	\end{align}
	Thus, the lemma is proved.
\end{proof}

Now, we are ready to prove  Fefferman's inequality in generalized Morrey spaces.

\begin{proof}[\textbf{Proof of Theorem \ref{th 2.1}}]
	Let $ 1 < \gamma <p $ and $ w :=[M(|V|^{\gamma})]^{\frac{1}{\gamma}}$. Then $ w  \in A_{1} \cap L^{p,\varphi} $
	according to Lemma \ref{lem 2.1}.
	First, we will show that \eqref{feff-in Morrey} holds for $ w $ in place of $V$.
	For any $ u \in C^{\infty}_{0}(\mathbb{R}^{n})$, let $ B $ be a ball such that $ u \in C^{\infty}_{0}(B)$.
	From the well-known inequality
	\begin{align}\label{th 2.2-2.1}
	|u(x)|\le C \int_{B_0}\frac{|\nabla u(y)|}{|x-y|^{n-1}} \ dy,
	\end{align}
	Tonelli's theorem, and Lemma \ref{lem 2.2}, we have
	\begin{align}\label{in th 2.1-14}
	\int_{\mathbb{R}^{n}} |u(x)|^{\alpha} w(x) dx
	& = \int_{B} |u(x)|^{\alpha} w(x) dx \nonumber \\
	& \leq C \int_{B} \left( \int_{B} \frac{|u(y)|^{\alpha-1} |\nabla u(y)| }{|x-y|^{n-1}} dy \right) |w(x)|dx \nonumber \\
	& \leq C \| w \|_{L^{p,\varphi}}^{\frac{1}{\alpha}} \int_{B} |u(x)|^{\alpha-1} |\nabla u(x)|
	[M(w)(x)]^{\frac{\alpha-1}{\alpha}}  dx.
	\end{align}
	H\"{o}lder's inequality and Lemma \eqref{lem 2.1} imply that
	\begin{align}\label{in th 2.1-15}
	\int_{B} |u(x)|^{\alpha-1} |\nabla u(x)|  [M(w)(x)]^{\frac{\alpha-1}{\alpha}} \, dx
	& \leq \left( \int_{B} |\nabla u(x)|^{\alpha} \, dx \right)^{\frac{1}{\alpha}} \nonumber \\
	& \qquad \times  \left( \int_{B} |u(x)|^{\alpha} M(w)(x) \, dx \right)^{\frac{\alpha-1}{\alpha}} \nonumber \\
	& \leq C \left( \int_{B} |\nabla u(x)|^{\alpha} \, dx \right)^{\frac{1}{\alpha}} \nonumber \\
	& \qquad \times \left( \int_{B} |u(x)|^{\alpha} w(x) \, dx \right)^{\frac{\alpha-1}{\alpha}}.
	\end{align}
	Substituting \eqref{in th 2.1-15} into \eqref{in th 2.1-14}, we obtain
	\begin{align*}
	\int_{\mathbb{R}^{n}} |u(x)|^{\alpha} |w(x)| dx
	& \leq C \| w \|_{L^{p,\varphi}}^{\frac{1}{\alpha}} \left( \int_{B} |\nabla u(x)|^{\alpha} \, dx
	\right)^{\frac{1}{\alpha}} \left( \int_{B} |u(x)|^{\alpha} w(x) \, dx \right)^{\frac{\alpha-1}{\alpha}}.
	\end{align*}
	Therefore
	\begin{align*}
	\int_{\mathbb{R}^{n}} |u(x)|^{\alpha} w(x) dx
	& \leq C \| w \|_{L^{p,\varphi}} \int_{B} |\nabla u(x)|^{\alpha} \, dx.
	\end{align*}
	By (\ref{f-Mf}), we have $|V(x)|=[|V(x)|^\gamma]^{\frac{1}{\gamma}}\le [M(|V(x)|^\gamma)]^{\frac{1}{\gamma}}=w(x)$.
	Hence, from the boundedness of the Hardy-Littlewood maximal operator on generalized Morrey spaces and Lemma \ref{lem 2.1}, we conclude that
	\begin{align*}
	\int_{\mathbb{R}^{n}} |u(x)|^{\alpha} |V(x)| dx
	& \leq \int_{\mathbb{R}^{n}} |u(x)|^{\alpha} w(x) dx \\
	& \leq C \| w \|_{L^{p,\varphi}} \int_{B} |\nabla u(x)|^{\alpha} \, dx \\
	& \leq C \| V \|_{L^{p,\varphi}} \int_{\mathbb{R}^{n}} |\nabla u(x)|^{\alpha} \, dx.
	\end{align*}
	This completes the proof.
\end{proof}

We have already  shown in Theorem \ref{th 2.1} that  Fefferman's inequality holds in generalized Morrey spaces
under certain conditions.

\subsection{Fefferman's Inequality in Stummel Classes}

We need the following lemma to prove Fefferman's inequality where its potentials belong to Stummel classes.
For the case $ \alpha = 2 $, this lemma  can also be deduced from property of the Riez kernel which is stated in \cite[p. 45]{L}.

\begin{lemma}\label{lem 2.3}
	Let $1 <\alpha \leq 2$ and $ \alpha < n$. For any ball $ B_{0} \subset \mathbb{R}^{n}$,
	the following inequality holds:
	\begin{equation*}
	\int\limits_{B_0}
	\frac{1}{|x-y|^{\frac{n-1}{\alpha-1}} |z-y|^{n-1}} \ dy
	\leq \frac{C}{|x-z|^{\frac{n-1}{\alpha-1}-1}}, \quad x,z \in B_{0}, \quad x \neq z.
	\end{equation*}
\end{lemma}

\begin{proof}
	Let $r := \frac{1}{2}|x-z|$. Then
	\begin{align}\label{lem 2.3-2.9}
	\int_{B_0}
	\frac{1}{|x-y|^{\frac{n-1}{\alpha-1}} |z-y|^{n-1}} \ dy
	& \leq \sum_{j=2}^{\infty} \int_{2^{j}r \leq |x-y|<2^{j+1}r}
	\frac{1}{|x-y|^{\frac{n-1}{\alpha-1}} |z-y|^{n-1}} \ dy \nonumber  \\
	&\qquad + \int_{|x-y|<4r}
	\frac{1}{|x-y|^{\frac{n-1}{\alpha-1}} |z-y|^{n-1}} \ dy  \nonumber \\
	& = I_{1} + I_{2}.
	\end{align}
	For $ I_{1}$, we get
	\begin{align}\label{lem 2.3-2.10}
	I_{1} & = \sum_{j=2}^{\infty} \int_{2^{j}r \leq |x-y|<2^{j+1}r}
	\frac{1}{|x-y|^{\frac{n-1}{\alpha-1}} |z-y|^{n-1}} \ dy \nonumber \\
	& \leq \sum_{j=2}^{\infty} \frac{1}{(2^{j}r)^{\frac{n-1}{\alpha-1}}} \int_{2^{j}r \leq |x-y|<2^{j+1}r}
	\frac{1}{|z-y|^{n-1}} \ dy.
	\end{align}
	Note that, $ 2^{j}r \leq |x-y|<2^{j+1}r $ implies $ 2^{j}r \leq |x-y| < 2r + |z-y|$.
	Therefore $ 2^{j-1}r \leq 2^{j}r-2r \leq |z-y| $. Hence the inequality \eqref{lem 2.3-2.10} becomes
	\begin{align}\label{lem 2.3-2.11}
	I_{1} & \leq \sum_{j=2}^{\infty} \frac{1}{(2^{j}r)^{\frac{n-1}{\alpha-1}}} \int_{2^{j}r \leq |x-y|<2^{j+1}r}
	\frac{1}{|z-y|^{n-1}} \ dy \nonumber \\
	& \leq C(n,\alpha) \sum_{j=2}^{\infty} \frac{1}{(2^{j}r)^{\frac{n-1}{\alpha-1}}} \frac{1}{(2^{j}r)^{n-1}}
	\int_{2^{j}r \leq |x-y|<2^{j+1}r} 1 \ dy \nonumber \\
	& \leq C(n,\alpha) \frac{1}{(r)^{\frac{n-1}{\alpha-1}-1}} \sum_{j=2}^{\infty} \frac{1}{(2^{j})^{\frac{n-1}{\alpha-1}-1}}.
	\end{align}
	Since $ \frac{n-1}{\alpha-1}-1>0$, the last series in \eqref{lem 2.3-2.11} is convergent. This gives us
	\begin{equation}\label{lem 2.3-2.12}
	I_{1} \leq C(n,\alpha) \frac{1}{(r)^{\frac{n-1}{\alpha-1}-1}}
	= \frac{C(n,\alpha)}{|x-z|^{\frac{n-1}{\alpha-1}-1}}.
	\end{equation}
	For $I_{2}$, we obtain
	\begin{align}\label{lem 2.3-2.13}
	I_{2}
	& = \int_{|x-y|<r}
	\frac{1}{|x-y|^{\frac{n-1}{\alpha-1}} |z-y|^{n-1}} \ dy
	+ \int_{r\leq|x-y|<4r}
	\frac{1}{|x-y|^{\frac{n-1}{\alpha-1}} |z-y|^{n-1}} \ dy \nonumber \\
	& \leq C(n,\alpha) \frac{1}{r^{\frac{n-1}{\alpha-1}-1}} =  \frac{C(n,\alpha)}{|x-z|^{\frac{n-1}{\alpha-1}-1}}.
	\end{align}
	Combining \eqref{lem 2.3-2.9}, \eqref{lem 2.3-2.12}, and \eqref{lem 2.3-2.13}, the lemma is proved.
\end{proof}

The following theorem is Fefferman's inequality where the potential belongs to a Stummel class.

\begin{proof}[\textbf{Proof of  Theorem \ref{th 2.2}}]
	The proof is separated into two cases, namely $ \alpha=1 $ and $ 1 < \alpha \leq 2$.
	We first consider the case $ \alpha=1$. Using the inequality \eqref{th 2.2-2.1}
	together with Fubini's theorem, we get
	\begin{align*}
	\int_{B_0} |u(x)| |V(x)|^p \ dx
	&\le C \int_{B_0} |\nabla u(y)|
	\int_{B_0} \frac{|V(x)|^p}{|x-y|^{n-1}} \ dx dy
	\\
	&\le
	C \int_{B_0} |\nabla u(y)|
	\int_{B(y, 2r_0)} \frac{|V(x)|^p}{|x-y|^{n-1}} \ dx dy.
	\end{align*}
	It follows from the last inequality and the doubling property of Stummel $p$-modulus of $V$ that
	\[
	\int_{B_0} |u(x)| |V(x)|^p \ dx
	\le
	C \,\eta_{\alpha,p}V(r_0)
	\int_{B_0} |\nabla u(x)| \ dx,
	\]
	as desired.
	
	We now consider the case $1<\alpha\le 2$. Using the inequality \eqref{th 2.2-2.1} and H\"older's inequality, we have
	\begin{align}\label{th 2.2-2.2}
	\int_{B_0} |u(x)|^\alpha |V(x)|^p \ dx
	&\le
	C \int_{B_0} |\nabla u(y)| \int_{B_0}
	\frac{|u(x)|^{\alpha-1}|V(x)|^p }{|x-y|^{n-1}} \  dx \ dy
	\nonumber \\
	&\le
	C \left(\int_{B_0} |\nabla u(y)|^\alpha \right)^{\frac{1}{\alpha}}
	\left(\int_{B_0} F(y)^{\frac{\alpha}{\alpha-1}} \ dy \right)^{\frac{\alpha-1}{\alpha}},
	\end{align}
	where $\displaystyle F(y):=\int_{B_0} \frac{|u(x)|^{\alpha-1}|V(x)|^p}{|x-y|^{n-1}} \ dx$, $\ y\in B_0$.
	Applying H\"older's inequality again, we have
	\begin{equation*}
	F(y)
	\le
	\left(
	\int_{B_0} \frac{|V(x)|^p}{|x-y|^{n-1}} \ dx
	\right)^{\frac{1}{\alpha}}
	\left(
	\int_{B_0} \frac{|u(z)|^\alpha|V(z)|^p}{|z-y|^{n-1}} \ dz
	\right)^{\frac{\alpha-1}{\alpha}},
	\end{equation*}
	so that
	\begin{align}\label{th 2.2-2.3}
	\int_{B_0} F(y)^{\frac{\alpha}{\alpha-1}} \ dy
	&\le
	\int_{B_0}
	\left( \int_{B_0} \frac{|V(x)|^{p}}{|x-y|^{n-1}}
	\ dx \right)^{\frac{1}{\alpha-1}}
	\int_{B_0} \frac{|u(z)|^\alpha|V(z)|^p}{|z-y|^{n-1}} \ dz \ dy \nonumber \\
	&=
	\int_{B_0}
	|u(z)|^\alpha |V(z)|^p G(z) \ dz,
	\end{align}
	where $\displaystyle G(z):=\int_{B_0} \left( \int_{B_0}
	\frac{|V(x)|^p}{|x-y|^{n-1} |z-y|^{(n-1)(\alpha-1)}} \ dx\right)^{\frac{1}{\alpha-1}} \ dy$,
	$z\in B_0$.
	By virtue of Minkowski's integral inequality (or Fubini's theorem for $\alpha=2$), we see that
	\begin{align}\label{th 2.2-2.4}
	G(z)^{\alpha-1}
	\le
	\int_{B_0} |V(x)|^p
	\left(
	\int_{B_0} \frac{1}{|x-y|^{\frac{n-1}{\alpha-1}} |z-y|^{n-1}} \ dy
	\right)^{\alpha-1} \ dx.
	\end{align}
	Combining \eqref{th 2.2-2.4}, doubling property of Stummel $p$-modulus of $V$, and the
	inequality in Lemma \ref{lem 2.1}, we obtain
	\begin{align}\label{th 2.2-2.5}
	G(z)
	\le C
	\left(\int_{B_0} \frac{|V(x)|^p}{|x-z|^{n-\alpha}} \ dx\right)^{\frac{1}{^{\alpha-1}}}
	\le C[\eta_{\alpha,p}V(r_0)]^{\frac{p}{\alpha-1}}.
	\end{align}
	Now, \eqref{th 2.2-2.3} and \eqref{th 2.2-2.5} give
	\begin{align}\label{th 2.2-2.6}
	\int_{B_0} |F(y)|^{\frac{\alpha}{\alpha-1}} \ dy
	\le C[\eta_{\alpha,p}V(r_0)]^{\frac{p}{{\alpha-1}}}
	\int_{B_0} |u(x)|^\alpha |V(x)|^p \ dx.
	\end{align}
	Therefore, from \eqref{th 2.2-2.2} and \eqref{th 2.2-2.6}, we get
	\begin{align}\label{th 2.2-2.7}
	\int_{B_0}& |u(x)|^\alpha  |V(x)|^p \ dx
	\nonumber	\\
	&\le C[\eta_{\alpha,p}V(r_0)]^{\frac{p}{\alpha}}
	\left(
	\int_{B_0} |\nabla u(x)|^\alpha \ dx
	\right)^{\frac{1}{\alpha}}
	\left(
	\int_{B_0}
	|u(x)|^\alpha |V(x)|^p \ dx
	\right)^{\frac{\alpha-1}{\alpha}}.
	\end{align}
	Dividing both sides by the third term of the right-hand side of \eqref{th 2.2-2.7}, we get the desired inequality.
\end{proof}

Let $ B $ be an open ball in $ \mathbb{R}^{n}$. If $u$ has weak gradient $\nabla u$ in $B$ and
$u$ is integrable over $B$, then we have the sub-representation inequality
\begin{equation}\label{sub-rep}
|u(x)-u_{B}| \leq C(n) \int\limits_{B} \frac{|\nabla u(y)|}{|x-y|^{n-1}} dy, \quad x \in B,
\end{equation}
where ${\displaystyle u_{B} := \frac{1}{|B|} \int_{B} u(y)dy}$.
Using the inequality \eqref{sub-rep} and the method in the proof of the previous theorem,
we obtain the following result.

\begin{theorem}
	Let $1 \leq p < \infty$, $1 \leq \alpha \le 2$, and $ \alpha < n$.
	Suppose that $u$ has weak gradient $\nabla u$ in $B_0:=B(x_{0},r_{0}) \subseteq \R^{n}$ and
	that $u$ is integrable over $B_0$.
	If $V \in \tilde{S}_{\alpha,p}$, then
	\begin{equation*}
	\int\limits_{B_0} |u(x)-u_{ B(x_{0},r_{0})}|^{\alpha} |V(x)|^{p} dx
	\leq
	C\left[\eta_{\alpha,p}V(r_{0})\right]^{p} \int\limits_{B_0} |\nabla u(x)|^{\alpha} dx,
	\end{equation*}
	where $C := C(n,\alpha)$.
\end{theorem}

\begin{remark}
	Note that the case $\alpha=2$ is exactly the Corollary 4.4 in \cite{CRR}.
\end{remark}

\section{Applications in Elliptic Partial Differential Equations}

The two lemmas below tell us that if a function vanishes of infinity order at some $ x_{0} \in \Omega $ and
fulfills the doubling integrability over some neighborhood of $ x_{0} $, then the function must be identically to zero in the neighborhood.

\begin{lemma}[\cite{GIA}]\label{lem 1.2}
	Let $ w \in L^{1}_{\rm loc}(\Omega) $ and $ B(x_{0},r) \subseteq \Omega$.
	Assume that there exists a constant $C>0$ satisfying
	\begin{equation*}
	\int\limits_{B(x_{0},r)} w(x)dx \leq C \int\limits_{B\left( x_{0},\frac{r}{2} \right) } w(x)dx.
	\end{equation*}
	If $w$ vanishes of infinity order at $x_{0}$, then $ w \equiv 0 $ in $ B(x_{0},r)$.
\end{lemma}

\begin{lemma}\label{lem 1.3}
	Let $ w \in L^{1}_{\rm loc}(\Omega) $ and $ B(x_{0},r) \subseteq \Omega$,
	and $ 0< \beta <1 $. Assume that there exists a constant $C>0$ satisfying
	\begin{equation*}
	\int\limits_{B(x_{0},r)} w^{\beta}(x)dx \leq C \int\limits_{B\left( x_{0},\frac{r}{2}\right) } w^{\beta}(x)dx.
	\end{equation*}
	If $w$ vanishes of infinity order at $x_{0}$, then $ w \equiv 0 $ in $ B(x_{0},r)$.
\end{lemma}

\begin{proof}
	According to the hypothesis, for every $ j \in \mathbb{N} $ we have
	\begin{align*}
	\int_{B(x_{0},r)} w^{\beta}(x)dx
	& \leq C^{1} \int_{B(x_{0},2^{-1}r)} w^{\beta}(x)dx \\
	& \leq C^{2} \int_{B(x_{0},2^{-2}r)} w^{\beta}(x)dx \\
	& \quad \vdots \\
	& \leq C^{j} \int_{B(x_{0},2^{-j}r)} w^{\beta}(x)dx.
	\end{align*}
	H\"{o}lder's inequality implies that
	\begin{align}\label{eq lem 1.3-1}
	\left( \int_{B(x_{0},r)} w^{\beta}(x)dx \right)^{\frac{1}{\beta}}
	\leq (C^{\frac{1}{\beta}})^{j} |B(x_{0},2^{-j}r)|^{\frac{1}{\beta}}
	\frac{|B(x_{0},2^{-j}r)|^{k}}{|B(x_{0},2^{-j}r)|^{k+1}} \int_{B(x_{0},2^{-j}r)} w(x)dx.
	\end{align}
	Now we choose $ k>0 $ such that $ C^{\frac{1}{\beta}}2^{-k} = 1$. Then \eqref{eq lem 1.3-1} gives
	\begin{align}\label{eq lem 1.3-2}
	\left( \int_{B(x_{0},r)} w^{\beta}(x)dx \right)^{\frac{1}{\beta}}
	\leq (v_{n}r^{n})^{\frac{1}{\beta}+k} (2^{-\frac{n}{\beta}})^{j} \frac{1}{|B(x_{0},2^{-j}r)|^{k+1}}
	\int_{B(x_{0},2^{-j}r)} w(x)dx,
	\end{align}
	where $ v_{n} $ is the Lebesgue measure of unit ball in $ \mathbb{R}^{n}$.
	Letting $ j \rightarrow \infty$, we obtain from \eqref{eq lem 1.3-2} that $ w^{\beta} \equiv 0 $
	on $ B(x_{0},r)$. Therefore $ w \equiv 0 $ on $ B(x_{0},r)$.
\end{proof}

The following lemma is used by many authors in working with elliptic
partial differential equation (for example, see \cite{CRR,CFG,CG,Z1}). This lemma and  the idea of its proof can be seen in \cite{JM}.
We state and give the proof of this lemma, since it had never been stated and proved formally  to the best of our knowledge.

\begin{lemma}\label{log of weak solution}
	Let $ w : \Omega \rightarrow \mathbb{R} $ and $ B(x,2r) $ be an open ball in $ \Omega $.
	If $ \log(w) \in BMO(B) $ with $ B=B(x,r) $, then there exist $ M>0 $ such that
	\begin{equation*}
	\int_{B(x,2r)} w^{\beta} dy \leq M^{\frac{1}{2}} \int_{B(x,r)} w^{\beta} dy
	\end{equation*}
	for some $ 0 < \beta < 1 $, or
	\begin{equation*}
	\int_{B(x,2r)} w \, dy \leq M^{\frac{1}{2}} \int_{B(x,r)} w \, dy.
	\end{equation*}
\end{lemma}

\begin{proof}
	By John-Nirenberg Theorem, there exist $\beta>0$ and $M>0$ such that
	\begin{equation}\label{eq 3.1}
	\left( \int_{B} \exp(\beta|\log(w)-\log(w)_{B}|) \, dy \right)^{2}  \leq M^{2}|B|^{2}.
	\end{equation}
	Assume that $ \beta <1$. Using \eqref{eq 3.1}, we compute
	\begin{align*}
	&\left(\int_{B} w^{\beta} dy \right) \left(\int_{B} w^{-\beta} dy \right)\\
	& =  \left( \int_{B} \exp(\beta \log(w)) dy\right) \left(  \int_{B} \exp(-\beta \log(w)) dy\right) \nonumber \\
	& = \left( \int_{B} \exp(\beta (\log(w)-\log(w)_{B})) dy \right) \left( \int_{B} \exp(-\beta (\log(w)-\log(w)_{B})) dy
	\right) \nonumber \\
	& \leq \left( \int_{B} \exp(\beta|\log(w)-\log(w)_{B}|) \, dy \right)^{2}  \leq M^{2} |B|^{2}, \nonumber
	\end{align*}
	which gives
	\begin{equation}\label{eq 3.2}
	\left(\int_{B} w^{-\beta} dy \right)^{\frac{1}{2}}
	\leq M |B| \left(\int_{B} w^{\beta} dy \right)^{-\frac{1}{2}}
	\leq  M |B| \left(\int_{B(x,2r)} w^{\beta} dy \right)^{-\frac{1}{2}}.
	\end{equation}
	Applying H\"{o}lder's inequality and \eqref{eq 3.2}, we obtain
	\begin{align}\label{eq 3.3}
	|B|
	& \leq \int_{B} w^{\frac{\beta}{2}} u^{-\frac{\beta}{2}} dy
	\leq \left(\int_{B} w^{\beta} dy \right)^{\frac{1}{2}} \left(\int_{B} w^{-\beta} dy \right)^{\frac{1}{2}} \nonumber \\
	& \leq M |B| \left(\int_{B} w^{\beta} dy \right)^{\frac{1}{2}}  \left(\int_{B(x,2r)} w^{\beta} dy \right)^{-\frac{1}{2}}.
	\end{align}
	From \eqref{eq 3.3}, we get
	\begin{equation*}
	\int_{B(x,2r)} w^{\beta} dy \leq M^{\frac{1}{2}} \int_{B(x,r)} w^{\beta} dy.
	\end{equation*}
	For the case $ \beta \geq 1$, we obtain from the inequality \eqref{eq 3.1} that
	\begin{align*}
	\left( \int_{B} \exp(|\log(w)-\log(w)_{B}|) \, dy \right)^{2}
	& \leq \left( \int_{B} \exp(\beta|\log(w)-\log(w)_{B}|) \, dy \right)^{2} \\
	& \leq M^{2}|B|^{2}.
	\end{align*}
	Processing the last inequality with previously method, we get
	\begin{equation*}
	\int_{B(x,2r)} w \, dy \leq M^{\frac{1}{2}} \int_{B(x,r)} w \, dy.
	\end{equation*}
	The proof is completed.
\end{proof}


Theorems \ref{th 2.1}  and \ref{th 2.2} are crucial in proving Theorem \ref{th 3.1}.
\begin{proof}[\textbf{Proof of Theorem \ref{th 3.1}}]
Given $ \delta > 0 $ and let $ \{ u_{k} \}_{k=1}^{\infty} $ be a sequence in $ C_{\rm 0}^{\infty}(\Omega) $ such that
$\lim_{k \rightarrow \infty} \| u_{k} - u \|_{H^{1}(\Omega)} = 0$. By taking a subsequence, we assume
$ u_{k} + \delta \rightarrow u + \delta $ a.e. on $ \Omega $ (see \cite[p.94]{B} or \cite[p.29]{EG}) and
$ u_{k} + \delta > 0 $ for all $ k \in \mathbb{N} $, since $ u \geq 0 $.

Let $\psi \in C_{0}^{\infty}(B(x,2r))$, $0 \leq \psi \leq 1$, $|\nabla \psi| \leq C_{1}r^{-1}$,
and $\psi := 1$ on $B(x,r)$. For every $ k \in \mathbb{N} $, we have $ \psi^{\alpha+1}/(u_{k}+\delta) \in H^{1}_{0}(\Omega) $.
Using this as a test function in the weak solution definition \eqref{b3-b}, we obtain
\begin{align}\label{eq 3.4}
\int\limits_{\Omega} \left\langle a \nabla u, \nabla (u_{k}+\delta) \right\rangle \frac{\psi^{\alpha+1}}{(u_{k}+\delta)^{2}}
&=(\alpha+1)\int\limits_{\Omega} \left\langle a\nabla u, \nabla \psi \right\rangle \frac{\psi^{\alpha}}{(u_{k}+\delta)} \nonumber \\
&+ \sum_{i= 1}^{n} \int\limits_{\Omega} b_{i} \frac{\partial u}{\partial x_{i}} \frac{\psi^{\alpha+1}}{(u_{k}+\delta)}
+ \int\limits_{\Omega} Vu\frac{\psi^{\alpha+1}}{(u_{k}+\delta)}.
\end{align}
Since ${\rm supp} (\psi) \subseteq B(x,2r)$, the inequality \eqref{eq 3.4} reduces to
\begin{align}\label{eq 3.5}
\int\limits_{B(x,2r)} \left\langle a \nabla u, \nabla (u_{k}+\delta) \right\rangle \frac{\psi^{\alpha+1}}{(u_{k}+\delta)^{2}}
& = (\alpha+1)\int\limits_{B(x,2r)} \left\langle a\nabla u, \nabla \psi \right\rangle \frac{\psi^{\alpha}}{(u_{k}+\delta)} \nonumber \\
& \qquad + \sum_{i= 1}^{n} \int\limits_{B(x,2r)} b_{i} \frac{\partial u}{\partial x_{i}} \frac{\psi^{\alpha+1}}{(u_{k}+\delta)} \nonumber \\
& \qquad + \int\limits_{B(x,2r)} Vu\frac{\psi^{\alpha+1}}{(u_{k}+\delta)}.
\end{align}
We will estimate all three terms on the right hand side of \eqref{eq 3.5}. For the first term,
according to \eqref{b2}, we have
\begin{equation}\label{eq 3.6}
|\left\langle a\nabla u, \nabla \psi \right\rangle | \leq \lambda^{-1}|\nabla u||\nabla \psi|.
\end{equation}
Combining  Young's inequality $ab \leq \epsilon a^{2} + \frac{1}{4\epsilon}b^{2} $ for every $\epsilon >0$ ($a,b>0$)
and the inequality \eqref{eq 3.6}, we have for every $\epsilon >0$
\begin{align}\label{eq 3.7}
(\alpha+1)\int\limits_{B(x,2r)} \left\langle a\nabla u, \nabla \psi \right\rangle \frac{\psi^{\alpha}}{(u_{k}+\delta)}
& \leq \epsilon \lambda^{-1}(\alpha+1) \int\limits_{B(x,2r)} \frac{|\nabla u|^{2}}{(u_{k}+\delta)^{2}} \psi^{2\alpha} \nonumber \\
& \qquad + \frac{\lambda^{-1}(\alpha+1)}{4\epsilon} \int\limits_{B(x,2r)} |\nabla \psi|^{2} \nonumber \\
& \leq \epsilon \lambda^{-1}(\alpha+1) \int\limits_{B(x,2r)} \frac{|\nabla (u+\delta)|^{2}}{(u_{k}+\delta)^{2}} \psi^{\alpha+1} \nonumber \\
& \qquad + \frac{\lambda^{-1}(\alpha+1)}{4\epsilon} \int\limits_{B(x,2r)} |\nabla \psi|^{2}.
\end{align}
To estimate the second term in \eqref{eq 3.5}, we use H\"{o}lder's inequality, Young's inequality and
Theorem \ref{th 2.1} or Theorem \ref{th 2.2}, to obtain
\begin{align}\label{eq 3.8}
\int\limits_{B(x,2r)} b_{i} \frac{\partial u}{\partial x_{i}} \frac{\psi^{\alpha+1}}{(u_{k}+\delta)}
& \leq \left( \int\limits_{B(x,2r)} \frac{|\nabla u|^{2}}{(u_{k}+\delta)^{2}} \psi^{\alpha+1}  \right)^{\frac{1}{2}}
\left( \int\limits_{B(x,2r)} b_{i}^{2} \psi^{\alpha+1} \right)^{\frac{1}{2}} \nonumber \\
& \leq \frac{\epsilon}{n} \int\limits_{B(x,2r)} \frac{|\nabla u|^{2}}{(u_{k}+\delta)^{2}} \psi^{\alpha+1} + \frac{1}{4n\epsilon}
\int\limits_{B(x,2r)} b_{i}^{2} \psi^{\alpha} \nonumber \\
& \leq \frac{\epsilon}{n} \int\limits_{B(x,2r)} \frac{|\nabla u|^{2}}{(u_{k}+\delta)^{2}} \psi^{\alpha+1} + \frac{1}{4n\epsilon}
C^{i}_{1} \int\limits_{B(x,2r)} |\nabla \psi|^{\alpha}.
\end{align}
for every $i=1,\dots,n$, where the constants $ C^{i}_{1}$'s depend on $ n,\alpha, \| b_{i}^{2} \|_{L^{p,\varphi}}$
or $ \eta_{\alpha}b_{i}^{2}(r_{0})$. From \eqref{eq 3.8} we have
\begin{align}\label{eq 3.9}
\sum_{i= 1}^{n} \int\limits_{B(x,2r)} b_{i} \frac{\partial u}{\partial x_{i}} \frac{\psi^{\alpha+1}}{(u_{k}+\delta)}
\leq \epsilon \int\limits_{B(x,2r)} \frac{|\nabla (u+\delta)|^{2}}{(u_{k}+\delta)^{2}} \psi^{\alpha+1} + \frac{1}{4\epsilon}
C_{2} \int\limits_{B(x,2r)} |\nabla \psi|^{\alpha},
\end{align}
where $C_{2}$ depends on $ \max\limits_{i} \{ C^{i}_{1} \}$. The estimate for the last term in \eqref{eq 3.5} is
\begin{align}\label{eq 3.10}
\int\limits_{B(x,2r)} Vu\frac{\psi^{\alpha+1}}{(u_{k}+\delta)}
\leq \int\limits_{B(x,2r)} V\frac{(u+\delta)}{(u_{k}+\delta)}\psi^{\alpha}.
\end{align}
Introducing \eqref{eq 3.7}, \eqref{eq 3.9}, and \eqref{eq 3.10} in \eqref{eq 3.5}, we get
\begin{align}\label{eq 3.11}
& \int\limits_{B(x,2r)} \left\langle a \nabla u, \nabla (u_{k}+\delta) \right\rangle \frac{\psi^{\alpha+1}}{(u_{k}+\delta)^{2}} \nonumber \\
& \leq \epsilon \lambda^{-1}(\alpha+1) \int\limits_{B(x,2r)} \frac{|\nabla (u+\delta)|^{2}}{(u_{k}+\delta)^{2}} \psi^{\alpha+1}
+ \frac{\lambda^{-1}(\alpha+1)}{4\epsilon} \int\limits_{B(x,2r)} |\nabla \psi|^{2} \nonumber \\
& \qquad + \epsilon \int\limits_{B(x,2r)} \frac{|\nabla (u+\delta)|^{2}}{(u_{k}+\delta)^{2}} \psi^{\alpha+1} + \frac{1}{4\epsilon}
C_{2} \int\limits_{B(x,2r)} |\nabla \psi|^{\alpha} + \int\limits_{B(x,2r)} V\frac{(u+\delta)}{(u_{k}+\delta)}\psi^{\alpha},
\end{align}
for every $ k \in \mathbb{N} $.

Since $ (u_{k} + \delta) \rightarrow (u+\delta) $ a.e. in $ \Omega $ and $ u + \delta > 0 $, then
\begin{equation}\label{eq 3.13}
\frac{1}{(u_{k} + \delta)} \rightarrow \frac{1}{(u+\delta)}, \, \text{a.e. in} \, \Omega.
\end{equation}
For $ j,i= 1, \dots, n $, we infer from \eqref{eq 3.13}
\begin{equation}\label{eq 3.14}
\frac{\partial (u+\delta)}{\partial x_{j}} \frac{\partial u}{\partial x_{i}} \frac{1}{(u_{k} + \delta)^{2}} \rightarrow \frac{\partial (u+\delta)}{\partial x_{j}} \frac{\partial u}{\partial x_{i}}\frac{1}{(u+\delta)^{2}}, \, \text{a.e. in} \, B(x,2r).
\end{equation}
For every $ k \in \mathbb{N} $, we have
\begin{equation}\label{eq 3.15}
\left| \frac{\partial (u+\delta)}{\partial x_{j}} \frac{\partial u}{\partial x_{i}} \frac{1}{(u_{k} + \delta)^{2}} \right|
\leq \left| \frac{\partial (u+\delta)}{\partial x_{j}} \right|  \left| \frac{\partial u}{\partial x_{i}} \right| \frac{1}{\delta^{2}},
\end{equation}
and
\begin{equation}\label{eq 3.16}
\int\limits_{B(x,2r)} \left| \frac{\partial (u+\delta)}{\partial x_{j}} \right|  \left| \frac{\partial u}{\partial x_{i}} \right| \frac{1}{\delta^{2}}
\leq \frac{1}{\delta^{2}} \left\| \frac{\partial u}{\partial x_{j}} \right\|_{L^{2}(\Omega)} \left\| \frac{\partial u}{\partial x_{i}} \right\|_{L^{2}(\Omega)} < \infty,
\end{equation}
since $ u \in H^{1}_{0}(\Omega) $. The properties \eqref{eq 3.14}, \eqref{eq 3.15}, and \eqref{eq 3.16} allow us to use the
Lebesgue Dominated Convergent Theorem to obtain
\begin{equation}\label{eq 3.17}
\lim\limits_{k \rightarrow \infty} \int\limits_{B(x,2r)} \left| \frac{\partial (u+\delta)}{\partial x_{j}} \frac{\partial u}{\partial x_{i}} \frac{1}{(u_{k} + \delta)^{2}} - \frac{\partial (u+\delta)}{\partial x_{j}} \frac{\partial u}{\partial x_{i}} \frac{1}{(u + \delta)^{2}}  \right| = 0.
\end{equation}
By H\"{o}lder's inequality,  we also have
\begin{align}\label{eq 3.18}
& \int\limits_{B(x,2r)} \left| \left( \frac{\partial (u_{k}+\delta)}{\partial x_{j}} - \frac{\partial (u+\delta)}{\partial x_{j}} \right)\frac{\partial u}{\partial x_{i}} \frac{1}{(u_{k}+\delta)^{2}}  \right| \nonumber \\
& \leq \frac{1}{\delta^{2}} \left\| \frac{\partial (u_{k}+\delta)}{\partial x_{j}} - \frac{\partial (u+\delta)}{\partial x_{j}}  \right\|_{L^{2}(B(x,2r))} \left\| \frac{\partial u}{\partial x_{i}} \right\|_{L^{2}(B(x,2r))} \nonumber \\
& \leq \frac{1}{\delta^{2}} \left\| \frac{\partial (u_{k})}{\partial x_{j}} - \frac{\partial u}{\partial x_{j}}  \right\|_{L^{2}(B(x,2r))} \left\| u \right\|_{H^{1}(\Omega)} \nonumber \\
&\leq \frac{1}{\delta^{2}} \left\| u_{k} - u  \right\|_{H^{1}(\Omega)} \left\| u \right\|_{H^{1}(\Omega)}
\end{align}
for all $ k \in \mathbb{N} $. Since $\lim_{k \rightarrow \infty} \| u_{k} - u \|_{H^{1}(\Omega)} = 0$, from \eqref{eq 3.18} we get
\begin{align}\label{eq 3.19}
\lim\limits_{k \rightarrow \infty} \int\limits_{B(x,2r)} \left| \left( \frac{\partial (u_{k}+\delta)}{\partial x_{j}} - \frac{\partial (u+\delta)}{\partial x_{j}} \right)\frac{\partial u}{\partial x_{i}} \frac{1}{(u_{k}+\delta)^{2}}  \right| = 0.
\end{align}
Note that
\begin{align}\label{eq 3.20}
& \int_{B(x,2r)} \left| a_{ij} \frac{\partial u}{\partial x_{i}} \frac{\partial (u_{k}+\delta)}{\partial x_{j}} \frac{\psi^{\alpha+1}}{(u_{k}+\delta)^{2}} - a_{ij} \frac{\partial u}{\partial x_{i}} \frac{\partial (u+\delta)}{\partial x_{j}} \frac{\psi^{\alpha+1}}{(u+\delta)^{2}} \right| \nonumber \\
& \leq \frac{\| a_{ij} \|_{L^{\infty}(\Omega)}}{\delta^{2}} \int_{B(x,2r)} \left| \frac{\partial u}{\partial x_{i}} \frac{\partial (u_{k}+\delta)}{\partial x_{j}} \frac{1}{(u_{k}+\delta)^{2}} - \frac{\partial u}{\partial x_{i}} \frac{\partial (u+\delta)}{\partial x_{j}} \frac{1}{(u+\delta)^{2}} \right| \nonumber \\
& \leq \frac{\| a_{ij} \|_{L^{\infty}(\Omega)}}{\delta^{2}} \int\limits_{B(x,2r)} \left| \frac{\partial (u+\delta)}{\partial x_{j}} \frac{\partial u}{\partial x_{i}} \frac{1}{(u_{k} + \delta)^{2}} - \frac{\partial (u+\delta)}{\partial x_{j}} \frac{\partial u}{\partial x_{i}} \frac{1}{(u + \delta)^{2}}  \right| \nonumber \\
& \qquad + \frac{\| a_{ij} \|_{L^{\infty}(\Omega)}}{\delta^{2}} \int\limits_{B(x,2r)} \left| \left( \frac{\partial (u_{k}+\delta)}{\partial x_{j}} - \frac{\partial (u+\delta)}{\partial x_{j}} \right)\frac{\partial u}{\partial x_{i}} \frac{1}{(u_{k}+\delta)^{2}}  \right|,
\end{align}
for all $ k \in \mathbb{N} $. Combining \eqref{eq 3.17}, \eqref{eq 3.19}, and \eqref{eq 3.20}, we have
\begin{align}\label{eq 3.12}
&\lim\limits_{k \rightarrow \infty} \int_{B(x,2r)} \left\langle a \nabla u, \nabla (u_{k}+\delta) \right\rangle \frac{\psi^{\alpha+1}}{(u_{k}+\delta)^{2}} \nonumber \\
& = \sum_{i, j = 1}^{n} \lim\limits_{k \rightarrow \infty} \int_{B(x,2r)} a_{ij}\frac{\partial u}{\partial x_{i}} \frac{\partial (u_{k}+\delta)}{\partial x_{j}} \frac{\psi^{\alpha+1}}{(u_{k}+\delta)^{2}} \nonumber \\
& = \sum_{i, j = 1}^{n} \int_{B(x,2r)} a_{ij} \frac{\partial u}{\partial x_{i}} \frac{\partial (u+\delta)}{\partial x_{j}} \frac{\psi^{\alpha+1}}{(u+\delta)^{2}} \nonumber \\
& = \int_{B(x,2r)} \left\langle a \nabla u, \nabla (u+\delta) \right\rangle \frac{\psi^{\alpha+1}}{(u+\delta)^{2}}.
\end{align}
From \eqref{eq 3.13},
\begin{equation}\label{eq 3.21}
\frac{|\nabla (u+\delta)|^{2}}{(u_{k}+\delta)^{2}} \psi^{\alpha+1} \rightarrow \frac{|\nabla (u+\delta)|^{2}}{(u+\delta)^{2}} \psi^{\alpha+1}, \, \text{a.e. in} \, B(x,2r).
\end{equation}
For every $ k \in \mathbb{N} $, we have
\begin{equation}\label{eq 3.22}
\frac{|\nabla (u+\delta)|^{2}}{(u_{k}+\delta)^{2}} \psi^{\alpha+1}
\leq \frac{1}{\delta^{2}}|\nabla (u+\delta)|^{2},
\end{equation}
and
\begin{equation}\label{eq 3.23}
\int\limits_{B(x,2r)} \frac{1}{\delta^{2}}  |\nabla (u+\delta)|^{2}
\leq \frac{1}{\delta^{2}} \| u \|_{H^{1}(\Omega)} < \infty,
\end{equation}
since $ u \in H^{1}_{0}(\Omega) $.
Therefore, by \eqref{eq 3.21}, \eqref{eq 3.22}, \eqref{eq 3.23}, and Lebesgue Dominated Convergence Theorem,
\begin{equation}\label{eq 3.24}
\lim\limits_{k \rightarrow \infty} \int\limits_{B(x,2r)} \frac{|\nabla (u+\delta)|^{2}}{(u_{k}+\delta)^{2}} \psi^{\alpha+1}
= \int\limits_{B(x,2r)} \frac{|\nabla (u+\delta)|^{2}}{(u+\delta)^{2}} \psi^{\alpha+1}.
\end{equation}
We also have
\begin{align}\label{eq 3.25}
V\frac{(u+\delta)}{(u_{k}+\delta)}\psi^{\alpha} \rightarrow V\frac{(u+\delta)}{(u+\delta)}\psi^{\alpha} = V\psi^{\alpha}, \text{a.e. in} \, B(x,2r)
\end{align}
because of \eqref{eq 3.13}. For every $ k \in \mathbb{N} $, we have
\begin{equation}\label{eq 3.26}
\left| V\frac{(u+\delta)}{(u_{k}+\delta)}\psi^{\alpha} \right|
\leq \frac{1}{\delta} |V| |u+\delta|.
\end{equation}
If the assumption \eqref{a.4} holds, then
\begin{align}\label{eq 3.27}
\int\limits_{B(x,2r)} \frac{1}{\delta} |V| |u+\delta|
& \leq \frac{1}{\delta} \int\limits_{B(x,2r)} |V| |u+\delta| \nonumber \\
& < \frac{1}{\delta} \left( \int\limits_{B(x,2r)} |V|^{2} \right)^{\frac{1}{2}} \left( \int\limits_{B(x,2r)} |u+\delta|^{2} \right)^{\frac{1}{2}}
< \infty,
\end{align}
since $ V \in L^{2}_{\rm loc}(\mathbb{R})$ and $ u \in H^{1}_{0}(\Omega) $. On the other hand, if the assumption \eqref{a.5} holds, then
$ V \in \tilde{S}_{\alpha} \subset \tilde{S}_{1}  $ by virtue to \cite[p.554]{THG}. Therefore, using Theorem \ref{th 2.2} we have
\begin{align}\label{eq 3.28}
\int\limits_{B(x,2r)} \frac{1}{\delta} |V| |u+\delta|
& \leq \frac{1}{\delta} \int\limits_{B(x,2r)} |V| + \frac{1}{4\delta} \int\limits_{B(x,2r)} |V| |u+\delta|^{2} \nonumber \\
& \leq \frac{1}{\delta} \int\limits_{B(x,2r)} |V| + \frac{1}{4\delta} C(n) \eta_{2}V(r) \int\limits_{B(x,2r)} |\nabla u|^{2} < \infty,
\end{align}
since $ u \in H^{1}_{0}(\Omega) $. Combining \eqref{eq 3.25}, \eqref{eq 3.26}, \eqref{eq 3.27} or \eqref{eq 3.28}, we can apply the
Lebesgue Dominated Convergence Theorem to have
\begin{align}\label{eq 3.30}
\lim\limits_{k \rightarrow \infty} \int\limits_{B(x,2r)} V\frac{(u+\delta)}{(u_{k}+\delta)}\psi^{\alpha}
= \int\limits_{B(x,2r)} V\psi^{\alpha}.
\end{align}
Theorem \ref{th 2.1} or Theorem \ref{th 2.2} allow us to get the estimate
\begin{align}\label{eq 3.31}
\int\limits_{B(x,2r)} V\psi^{\alpha} \leq C_{3} \int\limits_{B(x,2r)} |\nabla \psi|^{\alpha},
\end{align}
where the constant $C_{3}$ depends on $ n,\alpha$, and $ \| V \|_{L^{p,\varphi}}$ or $ \eta_{\alpha}V(r_{0})$.
Letting $ k \rightarrow \infty $ in \eqref{eq 3.11} and applying all informations in \eqref{eq 3.12}, \eqref{eq 3.24}, \eqref{eq 3.30},
and \eqref{eq 3.31}, we obtain
\begin{align}\label{eq 3.32}
& \int\limits_{B(x,2r)} \left\langle a \nabla u, \nabla (u+\delta) \right\rangle \frac{\psi^{\alpha+1}}{(u+\delta)^{2}} \nonumber \\
& \leq \epsilon \lambda^{-1}(\alpha+1) \int\limits_{B(x,2r)} \frac{|\nabla (u+\delta)|^{2}}{(u+\delta)^{2}} \psi^{\alpha+1}
+ \frac{\lambda^{-1}(\alpha+1)}{4\epsilon} \int\limits_{B(x,2r)} |\nabla \psi|^{2} \nonumber \\
& \qquad + \epsilon \int\limits_{B(x,2r)} \frac{|\nabla (u+\delta)|^{2}}{(u+\delta)^{2}} \psi^{\alpha+1} + \frac{1}{4\epsilon}
C_{2} \int\limits_{B(x,2r)} |\nabla \psi|^{\alpha} + C_{3} \int\limits_{B(x,2r)} |\nabla \psi|^{\alpha}.
\end{align}
Notice that
\begin{equation*}
\lambda \int\limits_{B(x,2r)} \frac{|\nabla (u+\delta)|^{2}}{(u+\delta)^{2}} \psi^{\alpha+1}
\leq \int\limits_{B(x,2r)} \left\langle a \nabla u, \nabla (u+\delta) \right\rangle \frac{\psi^{\alpha+1}}{(u+\delta)^{2}}
\end{equation*}
by the ellipticity condition \eqref{b2}. Moreover, by choosing $\epsilon := \frac{1}{2} \frac{\lambda^{2}}{ (\alpha+1)+1 }$,
the inequality \eqref{eq 3.32} is simplified by
\begin{align}\label{eq 3.33}
\int_{B(x,2r)} \frac{|\nabla (u+\delta)|^{2}}{(u+\delta)^{2}} \psi^{\alpha+1} \leq C_{4} \int_{B(x,2r)} |\nabla \psi|^{2} +
C_{5} \int_{B(x,2r)} |\nabla \psi|^{\alpha},
\end{align}
where the constant $C_{4}$ depends on $ \alpha $ and $ \lambda$, while the constant $ C_{5} $ depends on $ C_{2} $ and $ C_{3}$.
Therefore, \eqref{eq 3.33} implies
\begin{align*}
\int_{B(x,r)} |\nabla \log (u+\delta)|^{2}
& \leq \int_{B(x,2r)} \frac{|\nabla (u+\delta)|^{2}}{(u+\delta)^{2}} \psi^{\alpha+1} \\
& \leq C_{5} \int_{B(x,2r)} |\nabla \psi|^{2} + C_{6} \int_{B(x,2r)} |\nabla \psi|^{\alpha} \\
& \leq C \left( r^{-2}r^{n} + r^{-\alpha}r^{n} \right) = C r^{-2}r^{n}.
\end{align*}
The last constant $ C $ depends on $ C_{4}$ and $ C_{5}$. From H\"{o}lder's inequality,
\begin{equation*}
\left( \frac{1}{r^{n}} \int_{B(x,r)} |\nabla \log (u+\delta)|^{\alpha} \right)^{\frac{2}{\alpha}}
\leq \frac{1}{r^{n}} \int_{B(x,r)} |\nabla \log (u+\delta)|^{2}
\leq C r^{-2},
\end{equation*}
whence
\begin{equation}\label{eq 3.34}
\frac{1}{r^{n}} \int_{B(x,r)} |\nabla \log (u+\delta)|^{\alpha}
\leq C r^{-\alpha}.
\end{equation}
By using Poincar\'{e}'s inequality together with the inequality \eqref{eq 3.34}, the theorem is proved.
\end{proof}

By virtue of Theorem \ref{th 3.1}, we have the following corollary.

\begin{corollary}\label{cor 1}
	Let $u\geq0$ be a weak solution of $Lu=0$ and $B(x,2r) \subseteq \Omega$ where $ r \leq 1$.
	Then, for every $ \delta>0$, $ \log(u+\delta) \in BMO_{\alpha}(B(x,r))$.
\end{corollary}

Gathering Lemma \ref{lem 1.2}, Lemma \ref{lem 1.3}, Lemma \ref{log of weak solution}, and Corollary \ref{cor 1}, we obtain the
unique continuation property of the equation $ Lu=0 $ stated in Corollary \ref{cor 2}.

\begin{proof}[\textbf{Proof of the Corollary \ref{cor 2}}]
	Given $ x \in \Omega $ and let $B := B(x,r)$ be a ball where $B(x,2r) \subseteq \Omega$ and $ r \leq 1$.
Let $ \{ \delta_{j} \} $ be a sequence of real numbers in $ (0,1) $ which converges to $ 0 $. From Corollary \ref{cor 1}, we get
$ \log(u+\delta_{j}) \in BMO_{\alpha}(B) $. Therefore $ \log(u+\delta_{j}) \in BMO(B)$.
According to Lemma \ref{log of weak solution}, there exists a constant $ M>0 $ such that we have two cases:
\begin{equation*}
\int_{B(x,2r)} u^{\beta} dy \leq \int_{B(x,2r)} (u+\delta_{j})^{\beta} dy \leq M^{\frac{1}{2}} \int_{B(x,r)}
(u+\delta_{j})^{\beta} dy,
\end{equation*}
where $ 0<\beta <1$, or,
\begin{equation*}
\int_{B(x,2r)} u \, dy \leq \int_{B(x,2r)} (u+\delta_{j}) \, dy \leq M^{\frac{1}{2}} \int_{B(x,r)} (u+\delta_{j}) \, dy.
\end{equation*}
In both cases, letting $ j \rightarrow \infty $, we obtain
\begin{equation*}
\int_{B(x,2r)} u^{\beta} dy \leq M^{\frac{1}{2}} \int_{B(x,r)} u^{\beta} dy,
\end{equation*}
or,
\begin{equation*}
\int_{B(x,2r)} u \, dy \leq M^{\frac{1}{2}} \int_{B(x,r)} u \, dy.
\end{equation*}
Therefore, using Lemma \ref{lem 1.3} for the first case and Lemma \ref{lem 1.2} for the second case, $ u \equiv 0 $ in $ B(x,2r) $ if $u$ vanishes of infinity order at $x$.
\end{proof}

The example below shows that there exist an elliptic partial
differential which does not satisfies strong unique continuation property where its potential belongs to Morrey spaces $ L^{p,n-4p} $ and
$ \tilde{S}_{\beta} $ for all $ \beta \geq 4 $.

\begin{example}\label{ex1}
	Let $\Omega = B(0,1) \subseteq \mathbb{R}^{n} $, $ w: \Omega \rightarrow \mathbb{R} $ and $ V : \mathbb{R}^{n} \rightarrow \mathbb{R} $ which are defined by the formula
	\begin{equation*}
	w(x) =
	\begin{cases}
	\exp(-|x|^{-1})|x|^{-(n+1)}&, x \in \Omega \backslash \{ 0 \} \\
	1&, x = 0,
	\end{cases}
	\end{equation*}
	and
	\begin{equation*}
	V(x) =
	\begin{cases}
	3(n+1)|x|^{-2} - (n+5) |x|^{-3} + |x|^{-4}&, x \in \mathbb{R}^{n} \backslash \{ 0 \} \\
	0&, x = 0.
	\end{cases}
	\end{equation*}
	Note that, $ w(x)>0 $ for all $ x \in \Omega $. We will show that $ w $ vanishes of infinity order at $ x=0 $ 
and is a solution of Schr\"{o}dinger equation
	\begin{equation}\label{c.0}
	-\Delta w(x) + V(x)w(x) = 0, \qquad x \in \Omega.
	\end{equation}
	We calculate
	\begin{align*}
	\int_{|x|<r} w(x) dx & = \int_{|x|<r} \exp(-|x|^{-1})|x|^{-(n+1)} dx \\
	& = C(n) \exp\left( -\frac{1}{r}\right),
	\end{align*}
	by using polar coordinate. From the fact
	\begin{equation*}
	\lim_{r \rightarrow 0} \exp\left( -\frac{1}{r}\right) (r)^{-\gamma} = 0,
	\end{equation*}
	for all $ \gamma > 0 $, then
	\begin{equation*}
	\lim_{r \rightarrow 0} \frac{1}{|B(0,r)|^{\kappa}} \int_{|x|<r} w(x) dx
	= C(n) \lim_{r \rightarrow 0} \exp\left( -\frac{1}{r}\right) \frac{1}{r^{n\kappa}} = 0,
	\end{equation*}
	for all $ \kappa>0 $. Therefore $ w $ vanishes of infinity order at $ x=0 $. For $ i = 1, \dots, n $,
	\begin{align*}
	\frac{\partial w}{\partial x_{i}}(x) = w(x) \left( -(n+1)|x|^{-2} x_{i} + |x|^{-3} x_{i}  \right).
	\end{align*}
	Hence
	\begin{align*}
	\frac{\partial^{2} w}{\partial x_{i}^{2}}(x) = w(x) g(x),
	\end{align*}
	where
	\begin{align*}
	g(x )& = -(n+1)|x|^{-2} + |x|^{-3} + (2(n+1)+(n+1)^{2})|x|^{-4} x_{i}^{2} \\
	& \qquad - (3+2(n+1)) |x|^{-5} x_{i}^{2} + |x|^{-6} x_{i}^{2}.
	\end{align*}
	Consequently
	\begin{align*}
	\Delta w(x) & = \sum_{i= 1}^{n}\frac{\partial^{2} w}{\partial x_{i}^{2}}(x) = \sum_{i= 1}^{n} w(x)g(x) \\
	& = w(x) \sum_{i= 1}^{n} g(x) = w(x) \left(  3(n+1)|x|^{-2} - (n+5) |x|^{-3} + |x|^{-4} \right) \\
	& = w(x) V(x).
	\end{align*}
	This shows that $ w $ is the solution of the Schr\"{o}dinger equation. We conclude that this Schr\"{o}dinger equation
	does not satisfy the strong unique continuation property since the solution $ w $ vanishes of infinity order at $ x=0 $, but
	strictly positive function in $ \Omega $. Now, we will analyze the property of function $ V $.
	
	Let $ y \in \mathbb{R}^{n} $ and $ y \neq 0 $, we get
	\begin{align}\label{c.1}
	|V(y)| & \leq 3(n+1)|y|^{-2} + (n+5) |y|^{-3} + |y|^{-4}.
	\end{align}
	For every $ \beta > 4 $ and $ x \in \mathbb{R}^{n} $, then
	\begin{align}\label{c.2}
	\int_{|x-y|<r} \frac{|V(y)|}{|x-y|^{n-\beta}} dy
	& \leq C(n) \int_{|x-y|<r} \frac{|y|^{-2}}{|x-y|^{n-\beta}} dy + C(n) \int_{|x-y|<r} \frac{|y|^{-3}}{|x-y|^{n-\beta}} dy \nonumber \\
	& \qquad + C(n) \int_{|x-y|<r} \frac{|y|^{-4}}{|x-y|^{n-\beta}} dy,
	\end{align}
	by using \eqref{c.1}. For $ m \in \{2,3,4\} $, we estimate
	\begin{align}\label{c.3}
	\int_{|x-y|<r} \frac{|y|^{-m}}{|x-y|^{n-\beta}} dy
	& =  \int_{ \substack{ \{|x-y|<|y|\} \\ \cap \{|x-y|<r\} } } \frac{|y|^{-m}}{|x-y|^{n-\beta}} dy
	+ \int_{|y|\leq|x-y|<r} \frac{|y|^{-m}}{|x-y|^{n-\beta}} dy \nonumber \\
	& \leq \int_{|x-y|<r } \frac{|x-y|^{-m}}{|x-y|^{n-\beta}} dy + \int_{|y|<r} \frac{|y|^{-m}}{|y|^{n-\beta}} dy \nonumber \\
	& = C(n,m,\beta) r^{\beta-m}.
	\end{align}
	Introducing \eqref{c.3} in \eqref{c.2} we obtain
	\begin{align}\label{c.4}
	\int_{|x-y|<r} \frac{|V(y)|}{|x-y|^{n-\beta}} dy
	\leq C(n,\beta) (r^{\beta-2} + r^{\beta-3} + r^{\beta-4})
	\end{align}
	Since $ x $ arbitrary in \eqref{c.4}, then
	\begin{equation*}
	\eta_{\beta}V(r) \leq C(n,\beta) (r^{\beta-2} + r^{\beta-3} + r^{\beta-4})
	\end{equation*}
	which tell us $ \eta_{\beta}V(r) \rightarrow 0 $ for $ r \rightarrow 0 $.
	Whence $ V \in S_{\beta} \subseteq \tilde{S}_{\beta} $ for all $ \beta > 4 $. Moreover, $ \eta_{\beta}V(r) < \infty $
	for $ \beta = 4 $ and hence $ V \in \tilde{S}_{4} $.
	
	Let $ x \neq 0 $ and $ |x| \leq (n+5)^{-1} $, we have
	$ - (n+5) |x|^{-3} + |x|^{-4} \geq 0 $. This implies
	\begin{align}\label{c.5}
	V(x) \geq 3(n+1)|x|^{-2}.
	\end{align}
	Given $ 1 \leq \alpha \leq 2 $ and $ 0<r<(n+5)^{-1} $, then by \eqref{c.5}
	\begin{align}
	\eta_{\alpha}V(r)
	& \geq C(n) \int_{|y|<r} \frac{|y|^{-2}}{|y|^{n-\alpha}} dy \nonumber \\
	& \geq C(n) r^{\alpha-2} \int_{|y|<r} \frac{1}{|y|^{n}} dy = \infty.
	\end{align}
	Thus $ V \notin \tilde{S}_{\alpha} $.
	
	Define a function $ V^{*} = V \chi_{\Omega} $. Then $ V^{*}: \mathbb{R}^{n} \rightarrow \mathbb{R} $ and
	$ w $ is a solution the equation \eqref{c.0} where $ V $ is replaced by $ V^{*} $.
	For $ y \in \mathbb{R}^{n} $ and $ y \neq 0 $, we get
	\begin{align}\label{c.6}
	|V^{*}(y)| & \leq (4n+9)|y|^{-4}.
	\end{align}
	Given $ x \in \mathbb{R}^{n} $ and $ r>0 $. By \eqref{c.6} and using similar technique as in \eqref{c.3}, we have
	\begin{align}\label{c.7}
	\frac{1}{r^{n-4p}} \int_{|x-y|<r } |V^{*}(y)|^{p} dy
	\leq \frac{1}{r^{n-4p}} \int_{|x-y|<r } |y|^{-4p} dy
	= C(n,p).
	\end{align}
	According to \eqref{c.7}, we conclude $ V^{*} \in L^{p,n-4p} $.	 \qed
\end{example}

\begin{remark}
The equation $ Lu=0 $ has the strong unique continuation property if $ V, b^{2}_{i} \in \tilde{S}_{\alpha} $
for $ i=1, \dots, n $ and $ 1 \leq \alpha \leq 2 $ (see assumption \eqref{a.5}).
In view of Example \ref{ex1}, there exist $ V \in \tilde{S}_{\alpha}  $, $ \alpha \geq 4 $, and $ b_{i} = 0 $ for $ i=1, \dots, n $ such that the equation
$ Lu=0 $ does not have the strong unique continuation property.
However, the authors still do not know whether $ Lu=0 $ has the strong unique continuation property or not
if $ V, b^{2}_{i} \in \tilde{S}_{\alpha} $ for $ i=1, \dots, n $ and $ 2 < \alpha < 4 $.
\end{remark}

\begin{remark}
The equation $ Lu=0 $ has the strong unique continuation property if $ V, b^{2}_{i} \in L^{p,\varphi} $
where \eqref{a.4} holds. If we choose $ V \in L^{p,n-4p}  $ (i.e. $\alpha=4$) as in Example \ref{ex1} and $ b_{i} = 0 $ for $ i=1, \dots, n $,
then the equation
$ Lu=0 $ does not have the strong unique continuation property.
\end{remark}

\bigskip

\noindent{\bf Acknowledgement}. This research is supported by ITB Research \& Innovation
Program 2019. The first author also thanks LPDP Indonesia.

\bibliographystyle{amsplain}

\end{document}